\documentclass{article}
\usepackage{amsmath}
\usepackage{latexsym,amssymb}
\usepackage{amsmath}
\usepackage{hyperref}

\newtheorem{theorem}{Theorem}[section]
\newtheorem{lemma}[theorem]{Lemma}
\newtheorem{example}[theorem]{Example}

\newtheorem{corollary}[theorem]{Corollary}
\newtheorem{proposition}[theorem]{Proposition}

\newtheorem{definition}[theorem]{Definition}
\newenvironment{proof}{{\bf Proof:} }{$\Box$
\mbox{}}
\newenvironment{exs}{{\bf Examples:} }

\def \r {\vartriangleleft}

\def \As {\mathbf{As}}
\def \Conj {\mathbf{Conj}}

\def \G {\mathcal{G}}
\def \X {\mathcal{X}}

\input xypic
\input xy
\xyoption{v2}
\xyoption{all}
\xyoption{2cell}

\begin{document}

\title{Pullback Crossed Modules in the Category of Racks}

\author{
	Kad\.{i}r Em\.{i}r\thanks{Corresponding author.} \\
	\small{kadiremir86@gmail.com} 	\\
	\and Hatice G\"{u}ls\"{u}n Akay   \\
	\small{hgulsun@ogu.edu.tr} 
	\and \small{Department of Mathematics and Computer Science,} \\ \small{Eski\c{s}ehir Osmangazi University, Turkey.}}

\date{}
\maketitle

\begin{abstract}
In this paper, we define the pullback crossed modules in the category of racks which mainly based on a pullback diagram of rack morphisms with extra crossed module data on some of its arrows. Furthermore we prove that the conjugation functor, which is defined between the category of crossed modules of groups and of racks, preserves the pullback crossed modules.

\end{abstract}

AMS 2010 Classification: 18D05, 18A30, 18A40.

Keywords: Rack, crossed module, limit, pullback.

\section*{Introduction}

A rack $R$ is a set equipped with a non-associative binary
operation satisfying: $$\left( x\vartriangleleft y\right)
\vartriangleleft z=\left( x\vartriangleleft z\right) \vartriangleleft \left(
y\vartriangleleft z\right) $$ for all $x,y,z\in R$. Racks have been variously studied under plenty of names and a variety of terminology in the literature. They are called automorphic sets \cite{Brieskorn}, crystals \cite{Kauffman}, left distributive left quasigroup \cite{Stanosky}, and rack (as a modification of wrack) \cite{Conway}. The most important example of racks comes from the conjugation in a group $G$ where $g\vartriangleleft h=h^{-1}gh$ for all $g,h\in G$. This property yields to a functor $\mathbf{Conj \colon Grp \to Rack}$ from the category of groups to the category of racks. Moreover, there exists an adjunction \cite{RC1} between these two categories with:
\begin{align*}
	Hom\big(\As(X),G \big) \cong Hom\big(X , \Conj(G) \big)
\end{align*}
where the functor $\mathbf{As \colon Rack \to Grp}$ is left adjoint to the functor $\mathbf{Conj}$. 

\medskip

Crossed modules of racks \cite{Crans} generalizes the notion of crossed modules of groups \cite{WJHC} such that satisfying two certain Peiffer conditions. An interesting result of this notion is;  the functors $\mathbf{As}$ and $\mathbf{Conj}$ are preserving the crossed module structures, see \cite{Crans}. Therefore we can also consider them as the (induced) functors between the category of crossed modules of groups $\mathbf{XGrp}$ and the category of crossed modules of racks $\mathbf{XRack}$; where the previous adjunction leads to the following extended adjunction:
\begin{align*}
	Hom\big(\As_X(\mathcal{X}),\mathcal{G} \big) \cong Hom\big(\mathcal{X} , \Conj_X (\mathcal{G}) \big)
\end{align*}
when $\mathcal{X}$ is a crossed module of racks and $\mathcal{G}$ is a crossed module of groups. Consequently, one can say that the functor $\mathbf{Conj}_X$ preserves limits and $\mathbf{As}_X$ preserves colimits.

\medskip

 Crossed modules of groups or racks which have the same fixed codomain $A$ will be called as crossed$_A$ modules and lead to the full subcategories of the corresponding categories, which are denoted by $\mathbf{XGrp/A}$ and $\mathbf{XRack/A}$ respectively. Pullback crossed modules in the category of groups \cite{zbMATH00861611}, indeed are not the same as with the pullback objects in the category of crossed modules. They are constructed over a crossed$_R$ module and a group homomorphism $S \to R$ which gives rise to a crossed$_S$ module definition, in the sense of pullback diagram. This construction yields to the functor $i^{\star} \colon \mathbf{XRack/R \to XRack/S}$ which is a left adjoint to induced functor introduced in \cite{zbMATH00861611}.

\medskip

In this paper, we construct the pullback crossed module in the category of racks. Furthermore, we see that the functor $\mathbf{Conj}_X$ preserves this construction in the sense of the following commutative diagram:
$$\xymatrix@R=40pt@C=40pt{
	\mathbf{XGrp/_{\square}} \ar[d]^{i^{\star}} \ar[r]^{\Conj_X}  
	&  \mathbf{XRack/_{\square}} \ar[d]^{i^{\star}}    \\
	\mathbf{XGrp/_{\square}}  \ar[r]^{\Conj_X} 
	& \mathbf{XRack/_{\square}}        
}$$


%

\section*{Acknowledgment}

The authors are thankful to F.Wagemann for his valuable comments on the categorical aspects of racks. The second author is partially supported by T\"{u}bitak (the scientific and technological research council of Turkey).

\section{Preliminaries}

We recall some notions from \cite{Crans, RC1} which will be used in the sequel.

\subsection{Category of Racks}

\begin{definition}
A (right) rack $R$ is a set equipped with a (right) binary operation with the following conditions:
\begin{itemize}
	\item for each $a,b\in R$, there is a unique $c\in R$ such that:
	\begin{equation*}
	c\vartriangleleft a=b
	\end{equation*}
	\item for all $a,b,c\in R$, we have:
	\begin{equation*}
	\left( a\vartriangleleft b\right) \vartriangleleft c=\left(
	a\vartriangleleft c\right) \vartriangleleft \left( b\vartriangleleft
	c\right) .
	\end{equation*}
\end{itemize}

 A pointed rack is a rack $R$ with an element $1\in R$ such that (for all $a \in R$):
\begin{equation*}
1\vartriangleleft a=1\text{ \ \ \ \ and \ \ \ \ }a\vartriangleleft 1=a
\end{equation*}
From now on, all the racks will be pointed in the rest. 

\medskip

Let $R,S$ be two racks. A rack homomorphism is a map $f \colon R\rightarrow S $ such that:
\begin{equation*}
f\left( a \vartriangleleft b \right) =f\left( a\right)
\vartriangleleft f\left( b \right) \text{ \ \ \ \ and \ \ \ \ } f\left( 1\right) =1
\end{equation*}%
for all $a,b \in R$. Thus we have the category of racks, denoted by $\mathbf{Rack}$. Alternatively, for a point of view on racks where the two right and left rack operations are treated on an equal basis, see \cite{JFM}.
\end{definition}

\medskip

\begin{exs}
	\medskip
	
$\textbf{1)}$ Given a group $G$, there exists a rack structure on $G$ where the binary operation is:
\begin{equation*}
g\vartriangleleft h=h^{-1}gh,
\end{equation*}%
for all $g,h\in G$. This rack is called the conjugation rack of $G$, from which we get the functor:
\begin{align*}
\mathbf{Conj}: \textbf{Grp} \to \textbf{Rack}
\end{align*}

\medskip

$\textbf{2)}$ The core rack on a group $G$ is defined by:
\begin{equation*}
g\vartriangleleft h=hg^{-1}h,
\end{equation*}%
for all $g,h\in G$; however this construction is not functorial.

\medskip

$\textbf{3}$ Let $P,R$ be two racks, we have a rack structure on $P\times R$ defined by:
\begin{equation*}
\left( p,r\right) \vartriangleleft \left( p^{\prime },r^{\prime
}\right) =\left( p\vartriangleleft p^{\prime },r\vartriangleleft r^{\prime
}\right)
\end{equation*}
which is also the product object in the category of racks.
\end{exs}

\subsection{Rack Action}

\begin{definition}\label{act}
	Let $R,S$ be two racks. The map $\cdot \colon S \times R \to S$ is called a (right) action of $R$ on $S$ if it satisfies (for all $s,s' \in S$ and $r,r' \in R$):
	\begin{itemize}
	\item 	$ \left( s \cdot r\right) \cdot r^{\prime} =\left( s \cdot r^{\prime }\right) \cdot \left( r\vartriangleleft r^{\prime }\right)$,
	\item $(s \r s') \cdot r = (s \cdot r) \r (s' \cdot r)$.
	\end{itemize}
\end{definition}

\begin{definition}
	If there exists a (right) rack action of $R$ on $S$, the hemi-semi-direct product $S \rtimes R$ is the rack defined by:
	\begin{equation*}
		\left( s,r\right) \vartriangleleft \left( s^{\prime },r^{\prime }\right)
		=\left( s \cdot r' , r \r r' \right),
	\end{equation*}%
	for all $s,s' \in S$ and $r,r^{\prime }\in R$. 
\end{definition}

	Remark that $s'$ disappears in the hemi-semi direct operation which is the main technical difference from semi-direct product of groups and cause various problems when we deal with it.

\subsection{Crossed Modules of Racks}

A crossed module of racks is a rack homomorphism $\partial :R\to S$
together with a (right) rack action of $S$ on $R$ such that following two Peiffer relations hold (for all $r,r^{\prime }\in R$ and $s \in S$):
\begin{itemize}
\item[X1)] $ \partial \left( r\cdot s\right) =\partial \left( r\right) \vartriangleleft s$,
\item[X2)] $r\cdot \partial \left( r^{\prime }\right) =r\vartriangleleft r^{\prime }$.
\end{itemize}
If $\left( R,S,\partial \right) $ and $\left( R^{\prime },S^{\prime
},\partial ^{\prime }\right) $ are two crossed module of racks, a crossed module morphism:
\begin{equation*}
\left( f_{1},f_{0}\right) :\left( R,S,\partial \right) \rightarrow \left(
R^{\prime },S^{\prime },\partial ^{\prime }\right)
\end{equation*}%
is a tuple which consists of rack homomorphisms $f_{1}:R\rightarrow
R^{\prime }$, $f_{0}:S\rightarrow S^{\prime }$ such that: \newpage
\begin{itemize}
\item $\partial ^{\prime }f_{1}=f_{0} \, \partial$,
\item $f_{1}\left( r\cdot s\right) =f_{1}\left( r\right) \cdot
f_{0}\left( s\right)$,
\end{itemize}
for all $r\in R$, $s\in S$. Thus get the category of crossed modules of racks, denoted by $\mathbf{XRack}$. 

\bigskip

\begin{exs}
	
\medskip

$\mathbf{1)}$ Let $N \subset R$ be a normal subrack of $R$ (i.e. $n \r r \in R$ for all $n \in N$). The inclusion map $\partial :N\rightarrow R$ is a crossed module (inclusion crossed module) where the action is defined by the main rack operation.

\medskip

$\mathbf{2)}$ Let $\mu :M\rightarrow N$ be a crossed module of groups. We obtain a
crossed module of racks by passing to the associated conjugation racks of $M$
and $N$.$\medskip $

\end{exs}

\section{Fiber Product of Racks}

\begin{definition}
Let $\alpha  \colon P\to R$ and $\beta \colon S \to R$ be two rack homomorphisms. The fiber product $P\times _{R}S$ is the subrack of the rack $P\times S$ defined by:
\begin{equation*}
P\times _{R}S=\left\{ \left( p,s\right) \mid \alpha \left( p\right) =\beta
\left( s\right) \right\}
\end{equation*}%
In the categorical point of view, the fiber product is the equalizer of the parallel rack homomorphisms:
\begin{align*}
\xymatrix{
P \times S \ar@<0.5ex>[rr]^{\alpha \circ \pi _{1}} \ar@<-0.5ex>[rr]_{\beta \circ \pi _{2}} && R}
\end{align*} 
\end{definition}
%

\begin{proposition}
Let $\alpha \colon P\rightarrow R$ and $\beta :S \to R$ be two crossed modules of racks. Then the map $\partial :P\times _{R} S \to R$ defines a crossed module with the (right) rack action:
\begin{equation*}
\begin{array}{cll}
\left( P\times _{R} S\right) \times R & \rightarrow & P\times _{R}S \\
\left( \left( p,s\right) ,r\right) & \mapsto & \left( p,s\right) \cdot
r=\left( p\cdot r,s\cdot r\right)%
\end{array}%
\end{equation*}
\end{proposition}

\begin{proof}
$S$ acts on the fiber product $P\times _{R}S$ since (for all $\left(
p,s\right) ,\left( p^{\prime },s^{\prime }\right) \in P\times _{R}S$ and $%
r,r^{\prime }\in R$):
\begin{itemize}
	\item 
	\begin{align*}
	\left( \left( p,s\right) \cdot r\right) \cdot r^{\prime } & = \left(
	p\cdot r,s\cdot r\right) \cdot r^{\prime }   \\ 
	& = \left( \left( p\cdot r\right) \cdot r^{\prime },\left( s\cdot r\right)
	\cdot r^{\prime }\right)   \\ 
	& = \left( \left( p\cdot r^{\prime }\right) \cdot \left( r\vartriangleleft
	r^{\prime }\right) ,\left( s\cdot r^{\prime }\right) \cdot \left(
	r\vartriangleleft r^{\prime }\right) \right)  \\ 
	& = \left( \left( p\cdot r^{\prime }\right) ,\left( s\cdot r^{\prime
	}\right) \right) \cdot \left( r\vartriangleleft r^{\prime }\right)   \\ 
	& = \left( \left( p,s\right) \cdot r^{\prime }\right) \cdot \left(
	r\vartriangleleft r^{\prime }\right) 
	\end{align*}
	
	\item 
	\begin{align*}
	\left( \left( p,s\right) \vartriangleleft \left( p^{\prime },s^{\prime
	}\right) \right) \cdot r & =  \left( p\vartriangleleft p^{\prime
},s\vartriangleleft s^{\prime }\right) \cdot r   \\ 
& =  \left( \left( p\vartriangleleft p^{\prime }\right) \cdot r,\left(
s\vartriangleleft s^{\prime }\right) \cdot r\right)   \\ 
& =  \left( \left( p\cdot r\right) \vartriangleleft \left( p^{\prime }\cdot
r\right) ,\left( s\cdot r\right) \vartriangleleft \left( s^{\prime }\cdot
r\right) \right)   \\ 
& =  \left( \left( p\cdot r\right) ,\left( s\cdot r\right) \right)
\vartriangleleft \left( \left( p^{\prime }\cdot r\right) ,\left( s^{\prime
}\cdot r\right) \right)   \\ 
& =  \left( \left( p,s\right) \cdot r\right) \vartriangleleft \left( \left(
p^{\prime },s^{\prime }\right) \cdot r\right)  
\end{align*}
\end{itemize}
Also the map $\partial :P\times _{R}S\rightarrow R$ is a rack homomorphism
since: 
\begin{align*}
\partial \left( \left( p,s\right) \vartriangleleft \left( p^{\prime
},s^{\prime }\right) \right) & = \partial \left( p\vartriangleleft
p^{\prime },s\vartriangleleft s^{\prime }\right)   \\ 
& = \alpha \left( p\vartriangleleft p^{\prime }\right)  \\ 
& = \alpha \left( p\right) \vartriangleleft \alpha \left( p^{\prime
}\right)   \\ 
& = \partial \left( p,s\right) \vartriangleleft \partial \left( p^{\prime
},s^{\prime }\right) . 
\end{align*}
Finally $\partial $ is a crossed module of racks with the above action:
\begin{itemize}
	\item[X1)]
	\begin{align*}
	\partial \left( \left( p,s\right) \cdot r\right) & =  \partial \left(
	p\cdot r,s\cdot r\right)  \\ 
	& =  \alpha \left( p\cdot r\right)   \\ 
	& =  \alpha \left( p\right) \vartriangleleft r \quad (\because X1 \text{ condition of } \alpha)  \\ 
	& =  \partial \left( p,s\right) \vartriangleleft r 
	\end{align*}
	
	\item[X2)]
	\begin{align*}
	\left( p,s\right) \cdot \partial \left( p^{\prime },s^{\prime }\right) & = 
	\left( p,s\right) \cdot \alpha \left( p^{\prime }\right) \\ 
	& =  \left( p\cdot \alpha \left( p^{\prime }\right) ,s\cdot \alpha \left(
	p^{\prime }\right) \right)   \\ 
	& =  \left( p\cdot \alpha \left( p^{\prime }\right) ,s\cdot \beta \left(
	s^{\prime }\right) \right) \quad (\because \alpha (p') = \beta (s')) \\ 
	& =  \left( p\vartriangleleft p^{\prime },s\vartriangleleft s^{\prime
	}\right) \quad (\because X2 \text{ condition of } \alpha , \beta) \\ 
	& = \left( p,s\right) \vartriangleleft \left( p^{\prime },s^{\prime
	}\right) 
	\end{align*}
\end{itemize}
for all $\left( p,s\right) ,\left( p^{\prime },s^{\prime }\right) \in
P\times _{R}S$ and $r\in R$.
\end{proof}

\section{Pullback Crossed Modules in the Category of Racks}


\subsection{Idea}

\begin{definition}
Suppose that we have a crossed module of racks $\partial \colon P \to R$ and a rack homomorphism $\phi :S\rightarrow R$. The pullback crossed module of racks $\phi ^{\ast }(P,R,\partial
)=\left( \phi ^{\ast }(P),S,\partial ^{\ast }\right)$ is a crossed module of racks, such that satisfying the following universal property:
\begin{itemize}
\item For a given crossed module morphism of racks:
\begin{equation*}
\left( f,\phi \right) :\left( X,S,\mu \right) \rightarrow (P,R,\partial )
\end{equation*}
 there exists a unique crossed module morphism: 
 $$\left( f^{\ast
},id_{S}\right) :\left( X,S,\mu \right) \rightarrow \left( \phi ^{\ast
}(P),S,\partial ^{\ast }\right) $$ 
which makes the following diagram commutative:
$$ \xymatrix@R=40pt@C=40pt{
& & & (X,S,\mu ) \ar[d]^{(f,\phi)} \ar@{-->}[dlll]_{(f^{\ast },id_{S})} & \\
\ \left( \phi ^{\ast }(P),S,\partial^{\ast }\right) \ar[rrr]_{(\phi',\phi)} &  & & (P,R,\partial) & }
$$
\end{itemize}
On other words, the previous definition can be seen as a pullback diagram of rack homomorphisms:
\begin{align}\label{dia1}
\xymatrix@R=20pt@C=20pt{
  X \ar[dd]_{\mu} \ar[rr]^{f} \ar@{.>}[dr]|-{f^{\ast}}
              &  & P \ar[dd]^{\partial}  \\
&\phi^{\ast}(P)\ar[ur]_{\phi ^{\prime }}\ar[dl]^{\partial^{\ast}}&
 \\ S  \ar[rr]_{\phi}
               & & R             }
\end{align}
with extra crossed module data.
\end{definition}

\subsection{Construction}

Let $\partial \colon P \to R$ be a crossed module and let $\phi
:S\rightarrow R$ be a rack homomorphism. Define $\phi ^{\ast }(P) = P\times _{R}S$ and $\partial ^{\ast
} \colon \phi ^{\ast }(P)\rightarrow S$ by $\partial ^{\ast }\left( p,s\right) =s$.

\medskip

\textbf{Claim:} $\partial ^{\ast }$ is a crossed module where the action of $S$ on $\phi ^{\ast }(P)$ is defined by:
\begin{equation*}
\begin{array}{ccl}
\phi ^{\ast }(P)\times S & \rightarrow  & \phi ^{\ast }(P) \\
\left( \left( p,s\right) ,s^{\prime }\right)  & \mapsto  & \left( p,s\right)
\cdot s^{\prime }=\left( p\cdot \phi \left( s^{\prime }\right)
,s\vartriangleleft s^{\prime }\right)
\end{array}%
\end{equation*}%
First of all $\partial ^{\ast }$ is a rack homomorphism since:
\begin{align*}
\partial ^{\ast }\left( \left( p,s\right) \vartriangleleft \left( p^{\prime
},s^{\prime }\right) \right)  & =  \partial ^{\ast }\left(
p\vartriangleleft p^{\prime },s\vartriangleleft s^{\prime }\right)  \\
& =  s\vartriangleleft s^{\prime } \\
& =  \partial ^{\ast }\left( p,s\right) \vartriangleleft \partial ^{\ast
}\left( p^{\prime },s^{\prime }\right) 
\end{align*}
for all $\left( p,s \right) ,\left( p^{\prime },s' \right) \in \phi ^{\ast
}(P)$ and $s'' \in S$. Furthermore the action conditions are satisfied:
\begin{itemize}
	\item
	\begin{align*}
	\left( \left( p,s\right) \cdot s^{\prime }\right) \cdot s^{\prime \prime } & 
	= \left( p\cdot \phi \left( s^{\prime }\right) ,s\vartriangleleft
	s^{\prime }\right) \cdot s^{\prime \prime }   \\ 
	& =  \left( \left( p\cdot \phi \left( s^{\prime }\right) \right) \cdot \phi
	\left( s^{\prime \prime }\right) ,\left( s\vartriangleleft s^{\prime
	}\right) \vartriangleleft s^{\prime \prime }\right)    \\ 
	& =  \left( \left( p\cdot \phi \left( s^{\prime \prime }\right) \right)
	\cdot \left( \phi \left( s^{\prime }\right) \vartriangleleft \phi \left(
	s^{\prime \prime }\right) \right) ,\left( s\vartriangleleft s^{\prime \prime
	}\right) \vartriangleleft \left( s^{\prime }\vartriangleleft s^{\prime
	\prime }\right) \right)    \\ 
& =  \left( \left( p\cdot \phi \left( s^{\prime \prime }\right) \right)
\cdot \phi \left( s^{\prime }\vartriangleleft s^{\prime \prime }\right)
,\left( s\vartriangleleft s^{\prime \prime }\right) \vartriangleleft \left(
s^{\prime }\vartriangleleft s^{\prime \prime }\right) \right)    \\ 
& = \left( p\cdot \phi \left( s^{\prime \prime }\right) ,s\vartriangleleft
s^{\prime \prime }\right) \cdot \left( s^{\prime }\vartriangleleft s^{\prime
	\prime }\right)  \\ 
& =  \left( \left( p,s\right) \cdot s^{\prime \prime }\right) \cdot \left(
s^{\prime }\vartriangleleft s^{\prime \prime }\right)  
\end{align*}
	
	\item
\begin{align*}
\left( \left( p,s\right) \vartriangleleft \left( p^{\prime },s^{\prime
}\right) \right) \cdot s^{\prime \prime } & = \left( p\vartriangleleft
p^{\prime },s\vartriangleleft s^{\prime }\right) \cdot s^{\prime \prime } \\ 
& =  \left( \left( p\vartriangleleft p^{\prime }\right) \cdot \phi \left(
s^{\prime \prime }\right) ,\left( s\vartriangleleft s^{\prime }\right)
\vartriangleleft s^{\prime \prime }\right)  \\ 
& =  \left( \left( p\cdot \phi \left( s^{\prime \prime }\right)
\vartriangleleft p^{\prime }\cdot \phi \left( s^{\prime \prime }\right)
\right) ,\left( s\vartriangleleft s^{\prime \prime }\right) \vartriangleleft
\left( s^{\prime }\vartriangleleft s^{\prime \prime }\right) \right)  \\ 
& = \left( p\cdot \phi \left( s^{\prime \prime }\right) ,\left(
s\vartriangleleft s^{\prime \prime }\right) \right) \vartriangleleft \left(
p^{\prime }\cdot \phi \left( s^{\prime \prime }\right) ,\left( s^{\prime
}\vartriangleleft s^{\prime \prime }\right) \right)  \\ 
& = \left( \left( p,s\right) \cdot s^{\prime \prime }\right)
\vartriangleleft \left( \left( p^{\prime },s^{\prime }\right) \cdot
s^{\prime \prime }\right) 
\end{align*}
\end{itemize}
Finally, $\partial ^{\ast }$ is a crossed module:
\begin{itemize}
	\item[X1)]
\begin{align*}
\partial ^{\ast }\left( \left( p,s\right) \cdot s^{\prime }\right) & = 
\partial ^{\ast }\left( p\cdot \phi \left( s^{\prime }\right)
,s\vartriangleleft s^{\prime }\right) \\ 
& =  s\vartriangleleft s^{\prime } \\ 
& =  \partial ^{\ast }\left( p,s\right) \vartriangleleft s^{\prime }%
\end{align*}

\item[X2)]

\begin{align*}
\left( p,s\right) \cdot \partial ^{\ast }\left( p^{\prime },s^{\prime
}\right)  & = \left( p,s\right) \cdot s^{\prime }   \\ 
& = \left( p\cdot \phi \left( s^{\prime }\right) ,s\vartriangleleft
s^{\prime }\right)    \\ 
& = \left( p\cdot \partial \left( p^{\prime }\right) ,s\vartriangleleft
s^{\prime }\right)  \quad ( \because \partial \left( p^{\prime }\right) =\phi
\left( s^{\prime }\right) ) \\ 
& = \left( p\vartriangleleft p^{\prime },s\vartriangleleft s^{\prime
}\right)  \quad (\because X2 \text{ condition of } \partial) \\ 
& =  \left( p,s\right) \vartriangleleft \left( p^{\prime },s^{\prime
}\right)  & 
\end{align*}

\end{itemize}
for all $\left( p,s\right) ,\left( p^{\prime },s^{\prime }\right) \in \phi
^{\ast }(P)$.

\medskip

\textbf{Claim:} This construction satisfies the universal property. 

\medskip

To state it, we need the crossed module morphism:
\begin{equation*}
\left( \phi ^{\prime },\phi \right) :\left( \phi ^{\ast }(P),S,\partial
^{\ast }\right) \rightarrow (P,R,\partial )
\end{equation*}%
where $\phi ^{\prime }:\phi ^{\ast }(P)\rightarrow P$ is given by $\phi
^{\prime }\left( p,s\right) =p$. 

\medskip

Suppose that $\left( X,S,\mu \right) $ is an arbitrary crossed module with a crossed module morphism:
\begin{equation*}
\left( f,\phi \right) :\left( X,S,\mu \right) \rightarrow (P,R,\partial )
\end{equation*}

We need to prove that: there exists a unique crossed module morphism:
\begin{equation*}
\left( f^{\ast },id_{S}\right) :\left( X,S,\mu \right) \rightarrow \left(
\phi ^{\ast }(P),S,\partial ^{\ast }\right)
\end{equation*}%
such that:
\begin{equation*}
\left( \phi ^{\prime },\phi \right) \, \left( f^{\ast },id_{S}\right)  =\left(
f,\phi \right).
\end{equation*}

\medskip

Define $f^{\ast } \colon X\rightarrow \phi ^{\ast }(P)$ by $f^{\ast }(x)=\left(
f\left( x\right) ,\mu \left( x\right) \right) $, for all $x\in X$. Then the tuple $\left( f^{\ast },id_{S}\right) $ becomes a crossed module morphism, 
since (for all $s\in S$ and $x\in X$):
\begin{itemize}
\item
\begin{align*}
f^{\ast }\left( x\cdot s\right) & = \left( f\left( x\cdot s\right) ,\mu
\left( x\cdot s\right) \right) \\
& = \left( f\left( x\right) \cdot \phi \left( s\right) ,\mu
\left( x\cdot s\right) \right) \quad (\because (f,\phi) \text{ crossed module morphism})  \\
& = \left( f\left( x\right) \cdot \phi \left( s\right) ,\mu \left(
x\right) \vartriangleleft s\right) \quad (\because X1 \text{ condition of } \mu)  \\
& = \left( f\left( x\right) ,\mu \left( x\right) \right) \cdot s   \\
& = f^{\ast }(x)\cdot id_{S}\left( s\right) 
\end{align*}
\item
\begin{align*}
\partial ^{\ast }f^{\ast }(x) & =  \partial ^{\ast }\left( f\left( x\right)
,\mu \left( x\right) \right) \\
& = \mu \left( x\right)
\\
& =  id_{S} \, \mu \left( x\right) .
\end{align*}
\end{itemize}

Finally the diagram \eqref{dia1} 
commutes, since (for all $x \in X$):
\begin{align*}
\partial ^{\ast }f^{\ast }(x) & = \partial ^{\ast }\left( f\left( x\right)
,\mu \left( x\right) \right)  \\
& = \mu \left( x\right) \\
\phi ^{\prime }f^{\ast }(x) & = \phi ^{\prime }\left( f\left( x\right)
,\mu \left( x\right) \right)  \\
& = f\left( x\right) 
\end{align*}
and also $\phi \, \partial^{\ast} = \partial \, \phi'$ by the definition of $\phi^{\ast}(P)$.

\bigskip

Let $(f^{\prime }, id_S) \colon X\rightarrow \phi ^{\ast }(P)$ be a  crossed module morphism of racks with the same properties of $(f^{\ast } , id_S)$. Define  $f^{\prime }$ by $f^{\prime }\left(
x\right) =\left( p,s\right)$. Then we get:
\begin{align*}
\phi ^{\prime }f^{\prime }\left( x\right) & = f\left( x\right) \Leftrightarrow
\phi ^{\prime }\left( p,s\right) =f\left( x\right) \Leftrightarrow p=f\left(
x\right) \\
\partial ^{\ast }f^{\prime }\left( x\right) & =\mu \left( x\right)
\Leftrightarrow \partial ^{\ast }\left( p,s\right) =\mu \left( x\right)
\Leftrightarrow s=\mu \left( x\right) 
\end{align*}
lead to:
\[
f^{\prime }\left( x\right) =\left( p,s\right) =\left( f\left( x\right) ,\mu
\left( x\right) \right) =f^{\ast }(x)
\]
which implies that $(f^{\ast } , id_S)$ is unique, and completes the construction.

\begin{definition}
	Let us fix a rack $X$ as a codomain for all crossed modules and construct the related category which is the full subcategory of crossed modules of racks. These kinds of crossed modules will be called as crossed$_X$ modules and denote the corresponding category by $\mathbf{XRack/X}$. 
\end{definition}

\begin{corollary}
	As a consequence of the pullback crossed module structure in the category of racks, we have the functor: $$i^{\ast} \colon \mathbf{XRack/R \to XRack/S}.$$
\end{corollary}

%
%


\begin{example}
Let $\partial \colon N\rightarrow R$ be an inclusion crossed module and $\phi :S\rightarrow R$ be a rack homomorphism. Then the pullback crossed module is defined by:
\begin{align*}
\phi ^{\ast }\left( N\right) & =  \left\{ \left( n,s\right) \mid \partial
\left( n\right) =\phi \left( s\right) \text{, }n\in N\text{, }s\in S\right\}
\\
& \cong  \left\{ s\in S\mid \phi \left( s\right) =n,\text{ }n\in N\right\}
\\
& =  \phi ^{-1}\left( N\right)%
\end{align*}
with the diagram:
$$ \xymatrix @R=40pt@C=40pt{
\phi ^{-1}\left( N\right) \ar[r]^-{\phi ^{\prime }} \ar[d]_{\partial^{\ast }} &  N  \ar[d]^{\partial} \\
\ S \ar[r]_\phi & R  } $$
where the preimage $\phi ^{-1}\left( N\right) $ is a normal subrack of $S$.
\end{example}

It follows that:
\begin{example}
If $N=\left\{ 1\right\} $ and $R$ is a rack, then:
\begin{equation*}
\phi ^{\ast }\left( \left\{ 1\right\} \right) \cong \left\{ s\in S\mid \phi
\left( s\right) =1\right\} =\ker \phi .
\end{equation*}%
Thus $\left( \ker \phi ,S,\partial ^{\ast }\right) $ is a pullback crossed
module which implies $\ker \phi$ is a normal subrack. 
\end{example}

\begin{corollary}
	Kernel of a rack homomorphism is the particular case of pullback crossed module.
\end{corollary}

\begin{example}
If $N=R$ and $\phi $ is surjective, then:
\begin{equation*}
\phi ^{\ast }\left( R\right) =R\times S.
\end{equation*}
\end{example}

\section{Functorial Approach}

Let $R$ be a rack. The associated group $As(R)$ is the quotient of the free group $F(R)$ by the normal subgroup generated by the elements $y^{-1}x^{-1}y(x\r y)$ \cite{RC1}. This property leads to a functor:
\begin{align*}
\mathbf{As \colon Rack \to Grp}
\end{align*} 
which is right adjoint to the functor $\mathbf{Conj}$.

\medskip

The following lemma is due to \cite{RC1}:
\begin{lemma}
	Let $X$ be a rack and $G$ be a group. Given any rack homomorphism $f \colon X \to \Conj(G)$ there exists a unique group homomorphism $f_{\sharp} \colon \As(X) \to G$ such that the following diagram commutes:
	$$\xymatrix@R=40pt@C=40pt{
		X \ar[d]^{f} \ar[r]^{\mu}  
		&  \As(X) \ar[d]^{f_{\sharp}}    \\
		\Conj(G)  \ar[r]^{id} 
		& G     
	}$$
	where $\mu$ is the natural map. This leads to the adjunction:
	\begin{align*}
	Hom\big(\As(X),G \big) \cong Hom\big(X , \Conj(G) \big)
	\end{align*}
\end{lemma}

A great property of these $\As$ and $\Conj$ functors is: ``they both preserve the crossed module structure", proven in  \cite{Crans}. 
\begin{corollary}
 We have induced functors between the categories of such crossed modules:
\begin{align*}
	\As_X \colon \mathbf{XRack} \to \mathbf{XGrp} \quad \quad \mathbf{Conj}_X \colon \mathbf{XGrp} \to \mathbf{XRack}
\end{align*}
\end{corollary}

 Therefore we can give the following theorem which is the generalization of the above lemma:

\begin{theorem}
	Let $\X$ be a crossed module of racks and $\G$ be a crossed module of groups. Given a crossed module morphism of racks $(f,g) \colon \X \to \Conj_X(\G)$, there exists a unique crossed module morphism of groups $(f_{\sharp} , g_{\sharp} ) \colon \As_X(\X) \to \G$ such that the following diagram commutes:
		$$\xymatrix@R=40pt@C=40pt{
			\mathbf{\X} \ar[d]^{(f,g)} \ar[r]^{(\mu , \mu)}  
			&  \As_X(\X) \ar[d]^{(f_{\sharp} , g_{\sharp})}    \\
			\mathbf{Conj}_X(\G)  \ar[r]^{(id,id)} 
			& \mathbf{\G}        
		}$$
\end{theorem}
It follows that, we get a new extended adjunction between the categories of crossed modules of racks and category of crossed modules of groups:
\begin{align*}
Hom\big(\As_X (\mathcal{X}),\mathcal{G} \big) \cong Hom\big(\mathcal{X} , \Conj_X (\mathcal{G}) \big)
\end{align*}

As a main result of the paper, we have the following:

\begin{corollary}
The functor $\mathbf{Conj}_X$ preserves limits and $\mathbf{As}_X$ preserves colimits. Since the pullback crossed modules can be seen as a kind of pullback diagram which is also a certain case of categorical limits \cite{Adamek}, we have the following commutative diagram:
$$\xymatrix@R=40pt@C=40pt{
	\mathbf{XGrp/_{\square}} \ar[d]^{i^{\star}} \ar[r]^{\Conj_X}  
	&  \mathbf{XRack/_{\square}} \ar[d]^{i^{\star}}    \\
	\mathbf{XGrp/_{\square}}  \ar[r]^{\Conj_X} 
	& \mathbf{XRack/_{\square}}        
	}$$
\end{corollary}

\bibliographystyle{plain}
\bibliography{References}

\end{document}